\documentclass[11pt]{amsart}
\usepackage{tikz,amsthm,amsmath,amstext,amssymb,amscd,epsfig,euscript, mathrsfs, dsfont,pspicture,multicol,graphpap,graphics,graphicx,times,enumerate,subfig,sidecap,wrapfig,color}
\usepackage{amsmath}
\usepackage{amssymb}
\usepackage{pdfsync}

\newcommand{\R}{{\mathbb R}}       
\newcommand{\DD}{{\mathcal D}}

\newcommand{\HH}{{\mathcal H}}

\newcommand{\RR}{{\mathcal R}}

\newcommand{\EE}{{\mathcal E}}

\newcommand{\diam}{\mathop{\rm diam}}
\newcommand{\dist}{{\rm dist}}

\newcommand{\rf}[1]{{(\ref{#1})}}

\newcommand{\supp}{\operatorname{supp}}

\newcommand{\ve}{{\varepsilon}}
\newcommand{\vv}{{\vspace{2mm}}}
\newcommand{\vvv}{\vspace{4mm}}
\newcommand{\wt}[1]{{\widetilde{#1}}}

\newcommand{\bad}{{\mathsf{Bad}}}
\newcommand{\good}{{\mathsf{Good}}}

\newcommand{\HD}{{\mathsf{HD}}}
\newcommand{\LM}{{\mathsf{LM}}}

\def\Xint#1{\mathchoice
{\XXint\displaystyle\textstyle{#1}}%
{\XXint\textstyle\scriptstyle{#1}}%
{\XXint\scriptstyle\scriptscriptstyle{#1}}%
{\XXint\scriptscriptstyle\scriptscriptstyle{#1}}%
\!\int}
\def\XXint#1#2#3{{\setbox0=\hbox{$#1{#2#3}{\int}$ }
\vcenter{\hbox{$#2#3$ }}\kern-.58\wd0}}

\def\avint{\Xint-}

\newtheorem{theorem}{Theorem}[section]
\newtheorem{lemma}[theorem]{Lemma}
\newtheorem{corollary}[theorem]{Corollary}

\newtheorem*{lemma*}{Lemma}
\newtheorem*{theorem*}{Theorem}

\theoremstyle{definition}

\theoremstyle{remark}
\newtheorem{rem}[theorem]{\bf Remark}

\numberwithin{equation}{section}

\newcommand{\cnj}[1]{\overline{#1}}
\newcommand{\brem}{\begin{rem}}
\newcommand{\erem}{\end{rem}}

\def\d{\partial}

\begin{document}

\title{Rectifiability of harmonic measure in domains with porous boundaries}

\author{Jonas Azzam}
\address{Departament de Matem\`atiques\\ Universitat Aut\`onoma de Barcelona \\ Edifici C Facultat de Ci\`encies\\
08193 Bellaterra (Barcelona) }
\email{jazzam@mat.uab.cat}
\author{Mihalis Mourgoglou}
\address{Departament de Matem\`atiques\\ Universitat Aut\`onoma de Barcelona and Centre de Reserca Matem\` atica\\ Edifici C Facultat de Ci\`encies\\
08193 Bellaterra (Barcelona) }
\email{mmourgoglou@crm.cat}
\author{Xavier Tolsa}
\address{ICREA and Departament de Matem\`atiques\\ Universitat Aut\`onoma de Barcelona \\ Edifici C Facultat de Ci\`encies\\
08193 Bellaterra (Barcelona) }
\email{xtolsa@mat.uab.cat}
\keywords{Harmonic measure, absolute continuity, nontangentially accessible (NTA) domains, $A_{\infty}$-weights, doubling measures, porosity}
\subjclass[2010]{31A15,28A75,28A78}
\thanks{The three authors were supported by the ERC grant 320501 of the European Research Council (FP7/2007-2013). X.T. was also supported by 2014-SGR-75 (Catalonia), MTM2013-44304-P (Spain), and by the Marie Curie ITN MAnET (FP7-607647).}

\maketitle

\begin{abstract}
We show that if $n\geq 1$, $\Omega\subset \R^{n+1}$ is a connected domain that is porous around a subset $E\subset \d\Omega$ of finite and positive Hausdorff $\HH^{n}$-measure, and the harmonic measure $\omega$ is absolutely continuous with respect to $\HH^{n}$ on $E$, then $\omega|_E$ is concentrated on an $n$-rectifiable set. 
\end{abstract}

\section{Introduction}

There is a strong connection between the rectifiability of the boundary of a domain in Euclidean space and the possible absolute continuity of harmonic measure  with respect to Hausdorff measure. Recall that a set $E$ is {\it $n$-rectifiable} if it can be covered by a countable union of (possibly rotated) $n$-dimensional  Lipschitz graphs up to a set of zero $n$-dimensional Hausdorff measure $\HH^{n}$. The local F. and M. Riesz theorem of Bishop and Jones \cite{BJ} says that, if $\Omega$ is a simply connected planar domain and $\Gamma$ is a curve of finite length, then $\omega\ll \HH^{1}$ on $\d\Omega\cap \Gamma$, where $\omega$ stands for harmonic measure. In the same paper, Bishop and Jones also provide an example of a domain $\Omega$ whose boundary is contained in a curve of finite length, but $\HH^{1}(\d\Omega)=0<\omega(\d\Omega)$, thus showing that some sort of connectedness in the boundary is required. 

A higher dimensional version of the theorem of Bishop and Jones does not hold, even when the analogous ``connectivity" assumption holds for the boundary. In \cite{Wu}, Wu builds a topological sphere in $\R^{3}$ of finite surface measure bounding a domain whose harmonic measure charges a set of Hausdorff dimension 1 contained in $\R^{2}$. However, under some extra geometric assumptions, higher dimensional versions of the Bishop-Jones result do hold. For example, Wu shows in the same paper that if $\Omega\subset \R^{n+1}$ is a domain with {\it interior corkscrews}, meaning $\Omega\cap B(x,r)$ contains a ball of radius $r/C$ for every $x\in \d\Omega$ and $r\in (0,\diam\d\Omega)$, then $\omega\ll \HH^{n}$ on $\Gamma\cap \d\Omega$ whenever $\Gamma$ is a bi-Lipschitz image of $\R^{n}$ (or in fact a somewhat more general surface). 

Many results that establish absolute continuity follow a similar pattern to the results of Bishop, Jones, and Wu by considering portions of the boundary that are contained in nicer (and usually rectifiable) surfaces. For example, if $\Omega$ is a Lischitz domain (meaning he boundary is a union of Lipschitz graphs), then Dahlberg shows in \cite{Da} that $\omega \ll \HH^{n}\ll \omega$ on $\d\Omega$. The works of \cite{Ba} and \cite{DJ} also establish various degrees of mutual absolute continuity in nontangentially accessible domains when $\HH^{n}|_{\d\Omega}$ is Radon.
Recall that a domain is {\it nontangentially accessible (or NTA)} \cite{JK} if it has {\it exterior corkscrews} (meaning $(\cnj{\Omega})^{\circ}$ has interior corkscrews) and it is {\it uniform}, meaning there is $C>0$ so that for every $x,y\in \cnj{\Omega}$ there is a path $\gamma\subset\Omega$ connecting $x$ and $y$ such that
\begin{enumerate}[(a)]
\item the length of $\gamma$ is at most $C|x-y|$ and
\item for $t\in \gamma$, $\dist(t,\d\Omega)\geq \dist(t,\{x,y\})/C$. 
\end{enumerate}
 In \cite{Az}, the first author shows that, for NTA domains $\Omega\subset \R^{n+1}$, if $\Gamma\subset\partial \Omega$ is an $n$-Ahlfors regular set (meaning $\HH^{n}(B(x,r)\cap \Gamma)\sim r^{n}$ for any ball $B(x,r)$ centered on $\Gamma$ with $r\in (0,\diam \Gamma)$), then $\omega\ll \HH^{n}$ on $\d\Omega\cap \Gamma$ and $\omega|_{\d\Omega\cap \Gamma}$ is supported on an $n$-rectifiable set.
 
Without knowing that the portion of the boundary in question is contained in a nice enough surrogate set, things can go wrong. In \cite{AMT}, we constructed an NTA domain $\Omega\subset \R^{n+1}$ with very flat boundary, with $\HH^{n}|_{\d\Omega}$ locally finite and $n$-rectifiable, yet with $\d\Omega$ containing a set $E$ so that $\omega(E)>0=\HH^{n}(E)$. Observe that, in
this case, while $\d\Omega$ is still $n$-rectifiable, by the results of \cite{Az} described earlier, it follows that such a set $E$ cannot intersect a Lipschitz graph (or any Ahlfors regular set) in a set of positive $\omega$-measure. We think the result of 
\cite{AMT} is quite surprising in light of the previous results involving rectifiability and harmonic measure, as one might think that the rectifiability of $\partial\Omega$ should be enough to guarantee $\omega\ll \HH^{n}$. 

It is a natural question to ask then if the rectifiability of $\omega$ is actually {\it necessary} for absolute continuity, that is, if the support of $\omega$ can be exhausted up to a set of $\omega$-measure zero by $n$-dimensional Lipschitz graphs\footnote{We stress that when we speak of a measure $\omega$ being rectifiable, we mean that it may be covered up to a set of $\omega$-measure zero by $n$-dimensional Lipschitz {\it graphs}. This is a stronger criterion than rectifiability of measures as defined by Federer in \cite{Fed}, who defines this as being covered up to a set of $\omega$-measure zero by Lipschitz {\it images} of subsets of $\R^{n}$.}. Some results of this nature already exist. Recall that if $\Omega$ is a simply connected planar domain, $\phi:\mathbb{D}\rightarrow \Omega$ is a conformal map, and $G\subset \mathbb{D}$ is the set of points where $\phi$ has nonzero angular derivative, then there is $S\subset \d\Omega$ with $\HH^{1}(S)=0$ and $\omega(S\cup \phi(G))=1$ (see Theorem VI.6.1 in \cite{GM}). Thus, if $E\subset \d\Omega$ is such that $0<\HH^{1}(E)<\infty$ and $\omega\ll \HH^{1}$ on $E$, then $\omega(E\cap S)=0$, so $\omega$-almost every point in $E$ is in $\phi(G)$. Since all points in $\phi(G)$ are cone points  (p. 208 of \cite{GM}) and the set of cone points is a rectifiable set (Lemma 15.13 in \cite{Ma}), $\phi(G)\cap E$ is $1$-rectifiable and thus $\omega|_{E}$ is $1$-rectifiable.

In the work \cite{HMU}, Hofmann, Martell and Uriarte-Tuero show that if $\Omega\subset\R^{n+1}$ is a uniform domain, $\d\Omega$ is Ahlfors regular,
 and harmonic measure satisfies the weak-$A_{\infty}$ condition, then $\partial\Omega$ is uniformly rectifiable.
 The weak-$A_{\infty}$ condition is a stronger assumption than $\omega$ being absolutely continuous, but $\d\Omega$ being uniformly rectifiable is a stronger conclusion than just being rectifiable. We omit the definitions of these terms and refer the reader to these references. 
 \\

Our main result is the following.

%

\begin{theorem}\label{teo1}
Let $n\geq 1$ and $\Omega\subset\R^{n+1}$ be a proper domain of $\R^{n+1}$
and let $\omega$ be the harmonic measure in $\Omega$. 
Suppose that there exists $E\subset\partial\Omega$ with $0<\HH^n(E)<\infty$ and
that $\partial \Omega$ is porous in $E$, i.e. there is $r_{0}>0$ so that \
every ball $B$ centered at $E$ of radius at most $r_{0}$
contains another ball $B'\subset \R^{n+1}\setminus \partial\Omega$ with $r(B)\sim r(B')$, with the implicit constant
depending only on $E$.
If $\omega|_E$ is absolutely continuous with respect to $\HH^n|_{E}$, then
 $\omega|_E$ is $n$-rectifiable, in the sense that  $\omega$-almost all of $E$ can be covered by a countable union of $n$-dimensional (possibly rotated) Lipschitz graphs.
\end{theorem}
\vv

We list a few observations about this result:
\begin{enumerate}
\item Theorem \ref{teo1} is local: we don't assume $\HH^{n}|_{\d\Omega}$ is a Radon measure, only on the subset $E$.
\item We don't assume any strong connectedness property like uniformity, or a uniform exterior or interior corkscrew property, which the higher dimensional results mentioned earlier all rely upon. Aside from basic connectivity in Theorem \ref{teo1}, we only need a large ball in the complement of $\d\Omega$ in each ball centered on $E\subset\partial\Omega$ with no requirement whether that ball is in $\Omega$ or its complement. 
\item Examples of domains with porous boundaries are uniform domains, John domains, interior or exterior corkscrew domains, and the complement of an $n$-Ahlfors regular set. 
\item The theorem establishes rectifiability of the measure $\omega|_{E}$ and not of the set $E$: the set $E$ may very well contain a purely $n$-unrectifiable subset, but that subset must have $\omega$-measure zero. 


\item As far as we know, in the case $n=1$, the theorem is also new. 
\end{enumerate}
\vspace{3mm}

The following is an easy consequence of our main result.

\begin{corollary}\label{coro1}
Suppose that $n\geq 1$, $\Omega\subset \R^{n+1}$ is a connected domain, and $E\subset \d\Omega$ is a set such that $0<\HH^{n}(E)<\infty$, $\partial\Omega$ is porous in $E$, and $\HH^{n}\ll \omega$ on $E$. Then $E$ is $n$-rectifiable.
\end{corollary}

Indeed, by standard arguments, there is $E'\subset E$ such that $\HH^{n}(E\backslash E')=0$ and $\omega \ll \HH^{n}\ll \omega$ on $E'$. By Theorem \ref{teo1}, $\omega|_{E'}$ is $n$-rectifiable, but since $\HH^{n}\ll \omega$ on $E$, we also have that $E'$ is $n$-rectifiable, and thus $E$ is $n$-rectifiable.

\vv
We also mention that from Theorem \ref{teo1} in combination with the results of \cite{Az} we obtain the next corollary.

\begin{corollary}
Let $\Omega\subset \R^{n+1}$ be an NTA domain, $n\geq 1$, and let $E\subset \d\Omega$ be such that $0<\HH^{n}(E)<\infty$. Then $\omega|_{E}\ll \HH^{n}|_{E}$ if and only if $E$ may be covered by countably many $n$-dimensional Lipschitz graphs up to a set of $\omega$-measure zero.
\end{corollary}

The forward direction is just a consequence of Theorem \ref{teo1}, and the reverse direction follows from the result from \cite{Az} as described earlier since $n$-dimensional Lispchitz graphs are $n$-Ahlfors regular. 

\vv
During the preparation of this manuscript, we received a preprint by Hofmann and Martell \cite{HM} that shows that the
result from \cite{HMU} described above holds not only for uniform domains, but also for domains which
are complements of Ahlfors regular sets, again under the assumption that harmonic measure is weak-$A_{\infty}$. 
We  
thank Steve Hofmann for making his joint work available to us.
We remark that our method of proof of Theorem \ref{teo1} is completely independent of the techniques in \cite{HM} and previous works such as \cite{HMU}.
We also mention that after having written a first version of the present paper, Jos\'e Mar\'{\i}a Martell informed us that in a joint work with Akman, Badger and  Hofmann in preparation, they have obtained some result in the spirit of Corollary \ref{coro1} under some stronger assumptions (in particular, assuming $\partial \Omega$ to be Ahlfors regular). 


\vv
\section{Some notation}

We will write $a\lesssim b$ if there is $C>0$ so that $a\leq Cb$ and $a\lesssim_{t} b$ if the constant $C$ depends on the parameter $t$. We write $a\sim b$ to mean $a\lesssim b\lesssim a$ and define $a\sim_{t}b$ similarly.

For sets $A,B\subset \R^{n+1}$, we let 
\[\dist(A,B)=\inf\{|x-y|:x\in A,y\in B\}, \;\; \dist(x,A)=\dist(\{x\},A),\]
We denote the open ball of radius $r$ centered at $x$ by $B(x,r)$. For a ball $B=B(x,r)$ and $\delta>0$ we write $r(B)$ for its radius and $\delta B=B(x,\delta r)$. We let $U_\ve (A)$ to be the $\ve$-neighborhood of a set $A\subset \R^{n+1}$. For $A\subset \R^{n+1}$ and $0<\delta\leq\infty$, we set
\[\HH^{n}_{\delta}(A)=\inf\left\{\textstyle{ \sum_i \diam(A_i)^n: A_i\subset\R^{n+1},\,\diam(A_i)\leq\delta,\,A\subset \bigcup_i A_i}\right\}.\]
Define the {\it $n$-dimensional Hausdorff measure} as
\[\HH^{n}(A)=\lim_{\delta\downarrow 0}\HH^{n}_{\delta}(A)\]
and the {\it $n$-dimensional Hausdorff content} as $\HH^{n}_{\infty}(A)$. See Chapter 4 of \cite{Ma} for more details. \\

Given a signed Radon measure $\nu$ in $\R^{n+1}$ we consider the $n$-dimensional Riesz
transform
$$\RR\nu(x) = \int \frac{x-y}{|x-y|^{n+1}}\,d\nu(y),$$
whenever the integral makes sense. For $\ve>0$, its $\ve$-truncated version is given by 
$$\RR_\ve \nu(x) = \int_{|x-y|>\ve} \frac{x-y}{|x-y|^{n+1}}\,d\nu(y).$$
For $\delta\geq0$
 we set
$$\RR_{*,\delta} \nu(x)= \sup_{\ve>\delta} |\RR_\ve \nu(x)|.$$
We also consider the maximal operator
$$M^n_\delta\nu(x) = \sup_{r>\delta}\frac{|\nu|(B(x,r))}{r^n},$$
In the case $\delta=0$ we write $\RR_{*} \nu(x):= \RR_{*,0} \nu(x)$ and $M^n\nu(x):=M^n_0\nu(x)$.

\vv
\section{The strategy}

We fix a point $p\in\Omega$ far from the boundary to be specified later.
To prove that $\omega^p|_E$ is rectifiable we will show that any subset of positive harmonic measure 
of $E$ contains another subset $G$ of positive harmonic measure such that $\RR_*\omega^p(x)<\infty$ in $G$.
Applying a deep theorem essentially due to Nazarov, Treil and Volberg, one deduces that $G$ contains yet another
subset $G_0$ of positive harmonic measure such that $\RR_{\omega^p|_{G_0}}$ is bounded in $L^2(\omega^p|_{G_0})$. Then from
the results of Nazarov, Tolsa and Volberg in \cite {NToV} and \cite{NToV-pubmat}, it follows that $\omega^p|_{G_0}$ is $n$-rectifiable.
This suffices to prove the full $n$-rectifiability of $\omega^p|_E$.

One of the difficulties of Theorem \ref{teo1} is due to the fact that the non-Ahlfors  regularity of $\partial\Omega$ makes it difficult to apply some usual tools from potential of theory, such as the ones developed by Aikawa in \cite{Ai1} and \cite{Ai2}. In our proof we solve this issue by applying some stopping time arguments involving the harmonic measure and a suitable Frostman measure. 

The connection between harmonic measure and the Riesz transform is already used, at least implicitly, in the work of Hofmann, Martell and Uriarte-Tuero \cite{HMU}, and more explicitely in the paper by Hofmann, Martell and Mayboroda \cite{HMM}.
Indeed, in \cite{HMU}, in order to prove the uniform rectifiability of $\partial\Omega$,
the authors rely on the study of a square function related to the double gradient of the single layer potential and the application of an appropriate rectifiability criterion due to David and Semmes \cite{DS}. Note that the gradient of the single layer potential coincides with the Riesz transform away from the boundary.

We think that the Riesz transform is a much more flexible tool than the square function used in \cite{HMU}. Indeed, to work with the
Riesz transform with minimal regularity assumptions one can apply the techniques developed in the last so many years in the area of the so-called non-homogeneous Calder\'on-Zygmund theory. However, it is not clear to us if the aforementioned square function behaves
reasonably well without strong assumptions such as the $n$-Ahlfors regularity of $\partial \Omega$.

\vv

\section{Harmonic and Frostman measures}

We start by reviewing a result of Bourgain from \cite{Bo}. 

\begin{lemma}
\label{lembourgain}
There is $\delta_{0}>0$ depending only on $n\geq 1$ so that the following holds for $\delta\in (0,\delta_{0})$. Let $\Omega\subsetneq \R^{n+1}$ be a domain, $\xi \in \partial \Omega$, $r>0$, $B=B(\xi,r)$, and set $\rho:=\mathcal H_\infty^{s}(\partial\Omega\cap \delta B)/(\delta r)^{s}$ for some $s>n-1$. Then 
\[ \omega_{\Omega}^{x}(B)\gtrsim_{n} \rho\; \mbox{  for all }x\in \delta B.\]
\end{lemma}

\begin{proof}
We only prove the case $n\geq 2$, the $n=1$ case is similar, although one uses $-\log|\cdot|$ instead of $|\cdot|^{1-n}$ to define Green's function.

Without loss of generality, we assume $\xi=0$ and $r=1$. Let $\mu$ be a Frostman measure supported in $\delta B\cap \partial\Omega$ so that 
\begin{itemize}
\item $\mu(B(x,r))\leq r^{s}$ for all $x\in \R^{n+1}$ and $r>0$, 
\item  $\rho\delta^{s} \geq \mu(\delta B\cap \partial\Omega)\geq c\rho\delta^{s}$ where $c=c(n)>0$.
\end{itemize}

Define a function
\[
u(x)=\int\frac1{|x-y|^{n-1}}\,d\mu(y),\]
which is harmonic out of $\supp\mu$ and satisfies the following properties:
\begin{enumerate}[(i)]
\item For $x\in \delta B$,
\[ u(x)\geq 2^{1-n}\delta^{1-n}\mu(\delta B)\geq c2^{1-n}\delta^{s-n+1}\rho.\]
\item For $x\in \delta B$,
\[u(x)
\leq \sum_{j=0}^{\infty} \int_{\delta 2^{-j}\leq |x-y|<\delta 2^{-j+1}} \frac1{|x-y|^{n-1}}\,d\mu(x)
\leq \sum_{j=0}^{\infty}(2^{-j+1}\delta)^{s}(2^{-j}\delta)^{1-n}\sim \delta^{s-n+1}.\]
\item For $x\in B^{c}$,
\[u(x)=\int \frac1{|x-y|^{n-1}}\,d\mu(x)
\leq 2^{n-1}\mu(\delta B)\leq 2^{n-1}\rho\delta^{s}.\]
\item Thus, by the maximum principle, we have that $u(x)\lesssim \delta^{s-n+1}$ for all $x\in \R^{n+1}$. 
\end{enumerate}

Set 
\[ v(x)=\frac{u(x)-\sup_{\partial B}u}{\sup u}.\]
Then
\begin{enumerate}[(a)]
\item $v$ is harmonic in $(\delta B\cap \partial\Omega)^{c}$,
\item $v\leq 1$,
\item $v\leq 0$ on $B^{c}$,
\item for $x\in \delta B$ and $\delta$ small enough,
\[v(x) \gtrsim  \frac{c\delta^{s+1-n}\rho-2^{n-1}\rho\delta^{s}}{c\delta^{s-n+1}}
\gtrsim_\delta \rho.\]
\end{enumerate}
Let $\phi$ be any continuous compactly supported function equal to $1$ on $B$. Then $\int \phi d\omega_{\Omega}^{x}$ is at least any subharmonic function $f$ with $\limsup_{x\in \Omega \rightarrow \xi} f(x)\leq \phi(\xi)$. The function $v$ satisfies this, and so we have $\int \phi d\omega_{\Omega}^{x}\geq v(x)$. Taking the infimum over all such $\phi$, we get that $\omega_{\Omega}^{x}(B)\geq v(x)$, and the lemma follows.
%
\end{proof}
\vv

The proof of the next lemma is fairly standard but we include it for the sake of completeness. 

\begin{lemma}\label{l:w>G}
Let $\Omega\subsetneq \R^{n+1}$ be an open Greenian domain, $n\geq 1$, $\xi \in \partial\Omega$, $r>0$ and $B:=B(\xi,r)$.  Suppose that there exists a point $x_B \in \Omega$ so that the ball $B_0:=B(x_{B},r/C)$ satisfies $4B_0\subset \Omega\cap B$ for some $C>1$. Then, for $n\geq 2$,
\begin{equation}\label{eq:Green-lowerbound}
 \omega_{\Omega}^{x}(B)\gtrsim \omega_{\Omega}^{x_{B}}(B)\, r^{n-1}\, G_{\Omega}(x,x_{B})\,\, \,\,\,\,\,\,\text{for all}\,\, x\in \Omega\backslash  B_0.
 \end{equation}
In the case $n=1$, 
\begin{equation}\label{eq:Green-lowerbound2}
 \omega_{\Omega}^{x}(B)\gtrsim \omega_{\Omega}^{x_{B}}(B)\, \bigl|G_{\Omega}(x,x_{B})-
 G_{\Omega}(x,z)\bigr|\quad \mbox{for all $x\in\Omega\setminus B_0$ and $z\in \frac12B_0$.}
 \end{equation}

\end{lemma}

Note that the class of domains considered in Theorem \ref{teo1} are Greenian. Indeed, all open subsets of $\R^{n+1}$ are Greenian for $n\geq 2$ (Theorem 3.2.10 \cite{Hel}), and in the plane, if $\HH^{1}(\d\Omega)>0$, then $\d\Omega$ is nonpolar (p. 207 Theorem 11.14 of \cite{HKM}) and domains with nonpolar boundaries are Greenian by Myrberg's Theorem (see Theorem 5.3.8 on p. 133 of \cite{AG}). For the definitions of Greenian and polar sets, see \cite{Hel}.

\begin{proof}
First we consider the case $n\geq 2$.
Without loss of generality, we assume that $\omega^{x_B}(B)>0$ since otherwise \eqref{eq:Green-lowerbound} is trivial. We define a new domain $\Omega' := \Omega \setminus B_0 \subset \Omega$. From the definition of the Green function we have 
\begin{equation}\label{greenclaim}
G_{\Omega}(x,x_{B})\lesssim r(B)^{1-n} \mbox{ for }x\in \d B_{0}.
\end{equation}
 Since the set of Wiener irregular boundary points is polar (Corollary 4.5.5 \cite{Hel}), it holds that $G_\Omega(x,x_B)=0$ for all $x\in \partial \Omega' \cap \partial \Omega$ apart from a polar set. Moreover, for $x\in\partial B_0$ we have from \eqref{greenclaim} that
\begin{equation*}
G_\Omega(x,x_B) \leq c_0 \,\frac1{|x-x_B|^{n-1}} \leq c_1  \frac{\omega^x(B)}{r^{n-1\,}\omega^{x_B}(B)} ,
\end{equation*}
for some purely dimensional constant $c_1>0$, where the fact that $\omega^x(B)/\omega^{x_B}(B)\sim 1$ follows from the standard interior Harnack's inequality for $2B_0$.

Define now $u(x)=c_1 r^{1-n} \omega^x(B)/\omega^{x_B}(B) - G_\Omega(x,x_B) $ for all $x\in \Omega' \cup \partial \Omega'$, which is harmonic in $\Omega'$. Using that $G_\Omega(x,x_B) \lesssim |x-x_B|^{1-n} \lesssim r^{1-n}$ for any $x \in \Omega'$, we obtain that $u \geq - c_2 r^{1-n}$ in $\Omega'$. Therefore, by \cite[Theorem 4.2.21]{Hel}, in view of the fact that $u$ is harmonic and bounded below in $\Omega'$, $u \geq 0$ on $\partial \Omega'$ except for a polar set, and $\liminf_{|x| \to \infty} u(x) \geq 0$, we conclude \eqref{eq:Green-lowerbound}.

\vv

Now we deal with the case $n=1$. 
Again  we assume that $\omega^{x_B}(B)>0$ and we take $\Omega' = \Omega \setminus B_0 \subset \Omega$, as above. From the definition of the Green function, for $x\in\d B_0$ and $z\in \frac12B_0$ we have 
\begin{equation}\label{greenclaim*}
\bigl|G_{\Omega}(x,x_{B}) - G_{\Omega}(x,z)\bigr| = \left|\log\frac{|x-z|}{|x-x_B|} - \int \log\frac{|\xi-z|}{|\xi-x_B|}\,d\omega^x(\xi)\right|
\lesssim 1,
\end{equation}
since 
$$\frac{|x-z|}{|x-x_B|}\approx \frac{|\xi-z|}{|\xi-x_B|}\approx 1\quad \mbox{for $x\in\d B_0$, $\xi\in\d\Omega$, and $z\in
\frac12B_0$.}$$
Arguing as in the case $n\geq2$, we deduce that 
\begin{equation*}
\bigl|G_{\Omega}(x,x_{B}) - G_{\Omega}(x,z)\bigr| \leq c_0'  \leq c_1'  \frac{\omega^x(B)}{\omega^{x_B}(B)} ,
\end{equation*}
for some absolute constant $c_1'>0$, where the fact that $\omega^x(B)/\omega^{x_B}(B)\sim 1$ follows from the standard interior Harnack's inequality for $2B_0$.

For $x\in \Omega' \cup \partial \Omega'$ and a fixed $z\in\partial\frac12 B_0$, consider the function
$$u(x)=c_1' \frac{\omega^x(B)}{\omega^{x_B}(B)} - \bigl|G_\Omega(x,x_B)- G_{\Omega}(x,z)\bigr|.$$
This is superharmonic in $\Omega'$ and uniformly bounded. Therefore, since $u \geq 0$ on $\partial \Omega'$ except for a polar set and $\liminf_{|x| \to \infty} u(x) \geq 0$, we obtain \eqref{eq:Green-lowerbound2}.
\end{proof}

\vvv

From now on, $\Omega$ and $E$ will be as in Theorem \ref{teo1}. 
Also, fix a point $p\in\Omega$ and consider the harmonic measure $\omega^p$ of $\Omega$ with pole at $p$. 
The reader may think that $p$ is point deep inside $\Omega$.

The Green
function of $\Omega$ will be denoted just by $G(\cdot,\cdot)$.

Let $g\in L^1(\omega^p)$ be such that
$$\omega^p|_E = g\,\HH^n|_{\partial\Omega}.$$
Given $M>0$, let 
$$E_M= \{x\in\partial\Omega:M^{-1}\leq g(x)\leq M\}.$$
Take $M$ big enough so that $\omega^p(E_M)\geq \omega^p(E)/2$, say.
Consider an arbitrary compact set $F_M\subset E_M$ with $\omega^p(F_M)>0$. We will show that there exists $G_0\subset F_M$
with $\omega^p(G_0)>0$ which is $n$-rectifiable. Clearly, this suffices to prove that $\omega^p|_{E_M}$ is $n$-rectifiable,
and letting $M\to\infty$ we get the full $n$-rectifiability of $\omega^p|_E$.

Let $\mu$ be an $n$-dimensional Frostman measure for $F_M$. That is, $\mu$ is a non-zero Radon measure supported on $F_M$
such that 
$$\mu(B(x,r))\leq C\,r^n\qquad \mbox{for all $x\in\R^{n+1}$.}$$
Further, by renormalizing $\mu$, we can assume that $\|\mu\|=1$. Of course the constant $C$ above will depend on 
$\HH^n_\infty(F_M)$, and the same may happen for all the constants $C$ to appear,  but this will not bother us. Notice that $\mu\ll\HH^n|_{F_M}\ll \omega^p$. In fact, for any set $H\subset F_M$,
\begin{equation}\label{Frostman}
\mu(H)\leq C\,\HH^n_\infty(H)\leq C\,\HH^n(H)\leq C\,M\,\omega^p(H).
\end{equation}

\vv

\section{The dyadic lattice of David and Mattila}\label{secdya}

Now we will consider the dyadic lattice of cubes
with small boundaries of David-Mattila associated with $\omega^p$. This lattice has been constructed in \cite[Theorem 3.2]{David-Mattila} (with $\omega^p$ replaced by a general Radon measure). 
Its properties are summarized in the next lemma.

\begin{lemma}[David, Mattila]
\label{lemcubs}
Consider two constants $C_0>1$ and $A_0>5000\,C_0$ and denote $W=\supp\omega^p$. Then there exists a sequence of partitions of $W$ into
Borel subsets $Q$, $Q\in \DD_k$, with the following properties:
\begin{itemize}
\item For each integer $k\geq0$, $W$ is the disjoint union of the ``cubes'' $Q$, $Q\in\DD_k$, and
if $k<l$, $Q\in\DD_l$, and $R\in\DD_k$, then either $Q\cap R=\varnothing$ or else $Q\subset R$.
\vv

\item The general position of the cubes $Q$ can be described as follows. For each $k\geq0$ and each cube $Q\in\DD_k$, there is a ball $B(Q)=B(z_Q,r(Q))$ such that
$$z_Q\in W, \qquad A_0^{-k}\leq r(Q)\leq C_0\,A_0^{-k},$$
$$W\cap B(Q)\subset Q\subset W\cap 28\,B(Q)=W \cap B(z_Q,28r(Q)),$$
and
$$\mbox{the balls\, $5B(Q)$, $Q\in\DD_k$, are disjoint.}$$

\vv
\item The cubes $Q\in\DD_k$ have small boundaries. That is, for each $Q\in\DD_k$ and each
integer $l\geq0$, set
$$N_l^{ext}(Q)= \{x\in W\setminus Q:\,\dist(x,Q)< A_0^{-k-l}\},$$
$$N_l^{int}(Q)= \{x\in Q:\,\dist(x,W\setminus Q)< A_0^{-k-l}\},$$
and
$$N_l(Q)= N_l^{ext}(Q) \cup N_l^{int}(Q).$$
Then
\begin{equation}\label{eqsmb2}
\omega^p(N_l(Q))\leq (C^{-1}C_0^{-3d-1}A_0)^{-l}\,\omega^p(90B(Q)).
\end{equation}
\vv

\item Denote by $\DD_k^{db}$ the family of cubes $Q\in\DD_k$ for which
\begin{equation}\label{eqdob22}
\omega^p(100B(Q))\leq C_0\,\omega^p(B(Q)).
\end{equation}
We have that $r(Q)=A_0^{-k}$ when $Q\in\DD_k\setminus \DD_k^{db}$
and
\begin{equation}\label{eqdob23}
\omega^p(100B(Q))\leq C_0^{-l}\,\omega^p(100^{l+1}B(Q))\quad
\mbox{for all $l\geq1$ such that $100^l\leq C_0$ and $Q\in\DD_k\setminus \DD_k^{db}$.}
\end{equation}
\end{itemize}
\end{lemma}

\vv

We use the notation $\DD=\bigcup_{k\geq0}\DD_k$. Observe that the families $\DD_k$ are only defined for $k\geq0$. So the diameter of the cubes from $\DD$ are uniformly
bounded from above.
We set
$\ell(Q)= 56\,C_0\,A_0^{-k}$ and we call it the side length of $Q$. Notice that 
$$\frac1{28}\,C_0^{-1}\ell(Q)\leq \diam(Q)\leq\ell(Q).$$
Observe that $r(Q)\sim\diam(Q)\sim\ell(Q)$.
Also we call $z_Q$ the center of $Q$, and the cube $Q'\in \DD_{k-1}$ such that $Q'\supset Q$ the parent of $Q$.
 We set
$B_Q=28 \,B(Q)=B(z_Q,28\,r(Q))$, so that 
$$W\cap \tfrac1{28}B_Q\subset Q\subset B_Q.$$

We assume $A_0$ big enough so that the constant $C^{-1}C_0^{-3d-1}A_0$ in 
\rf{eqsmb2} satisfies 
$$C^{-1}C_0^{-3d-1}A_0>A_0^{1/2}>10.$$
Then we deduce that, for all $0<\lambda\leq1$,
\begin{align}\label{eqfk490}\nonumber
\omega^p\bigl(\{x\in Q:\dist(x,W\setminus Q)\leq \lambda\,\ell(Q)\}\bigr) + 
\omega^p\bigl(\bigl\{x\in 3.5B_Q:\dist&(x,Q)\leq \lambda\,\ell(Q)\}\bigr)\\
&\leq
c\,\lambda^{1/2}\,\omega^p(3.5B_Q).
\end{align}

We denote
$\DD^{db}=\bigcup_{k\geq0}\DD_k^{db}$.
Note that, in particular, from \rf{eqdob22} it follows that
\begin{equation}\label{eqdob*}
\omega^{p}(3B_{Q})\leq \omega^p(100B(Q))\leq C_0\,\omega^p(Q)\qquad\mbox{if $Q\in\DD^{db}.$}
\end{equation}
For this reason we will call the cubes from $\DD^{db}$ doubling. 

As shown in \cite[Lemma 5.28]{David-Mattila}, every cube $R\in\DD$ can be covered $\omega^p$-a.e.\
by a family of doubling cubes:
\vv

\begin{lemma}\label{lemcobdob}
Let $R\in\DD$. Suppose that the constants $A_0$ and $C_0$ in Lemma \ref{lemcubs} are
chosen suitably. Then there exists a family of
doubling cubes $\{Q_i\}_{i\in I}\subset \DD^{db}$, with
$Q_i\subset R$ for all $i$, such that their union covers $\omega^p$-almost all $R$.
\end{lemma}

The following result is proved in \cite[Lemma 5.31]{David-Mattila}.
\vv

\begin{lemma}\label{lemcad22}
Let $R\in\DD$ and let $Q\subset R$ be a cube such that all the intermediate cubes $S$,
$Q\subsetneq S\subsetneq R$ are non-doubling (i.e.\ belong to $\DD\setminus \DD^{db}$).
Then
\begin{equation}\label{eqdk88}
\omega^p(100B(Q))\leq A_0^{-10n(J(Q)-J(R)-1)}\omega^p(100B(R)).
\end{equation}
\end{lemma}


Given a ball $B\subset \R^{n+1}$, we consider its $n$-dimensional density:
$$\Theta_\omega(B)= \frac{\omega^p(B)}{r(B)^n}.$$

From the preceding lemma we deduce:

\vv
\begin{lemma}\label{lemcad23}
Let $Q,R\in\DD$ be as in Lemma \ref{lemcad22}.
Then
$$\Theta_\omega(100B(Q))\leq C_0\,A_0^{-9n(J(Q)-J(R)-1)}\,\Theta_\omega(100B(R))$$
and
$$\sum_{S\in\DD:Q\subset S\subset R}\Theta_\omega(100B(S))\leq c\,\Theta_\omega(100B(R)),$$
with $c$ depending on $C_0$ and $A_0$.
\end{lemma}

For the easy proof, see
 \cite[Lemma 4.4]{Tolsa-memo}, for example.

\vv
From now on we will assume that $C_0$ and $A_0$ are some big fixed constants so that the
results stated in the lemmas of this section hold. Further, we will choose the pole $p\in\Omega$ of the harmonic measure $\omega^p$ so that $\dist(p,\,\d \Omega) \geq 10C_0$. The existence of such point $p$ can be 
assumed by dilating $\Omega$ by a suitable factor if necessary.

\vv

\section{The bad cubes}\label{secbad}

Now we need to define a family of bad cubes.
We say that $Q\in\DD$ is bad and we write $Q\in\bad$, if $Q\in\DD$ is a maximal cube satisfying one of the conditions below: 
\begin{itemize}
\item[(a)] $\mu(Q)\leq \tau\,\omega^p(Q)$, where $\tau>0$ is a small parameter to be fixed below, or
\item[(b)]  $\omega^p(3B_Q)\geq A\,r(B_Q)^n$, where $A$ is some big constant to be fixed below.
\end{itemize}
The existence maximal cubes is guarantied by the fact that all the cubes from $\DD$ have side length uniformly bounded from
above (since $\DD_k$ is defined only for $k\geq0$). 
If the condition (a) holds, we write $Q\in\LM$ (little measure $\mu$) and in the case (b), $Q\in\HD$ (high density).
On the other hand, if a cube $Q\in\DD$ is not contained in any cube from $\bad$, we say that $Q$ is good and we write
$Q\in\good$.
 
Notice that 
$$ 
\sum_{Q\in\LM} \mu(Q) \leq \tau \sum_{Q\in\LM} \omega^p(Q) \leq \tau\,\|\omega\|=\tau=\tau\,\mu(F_M).$$
Therefore, taking into account that $\tau\leq1/2$ and that $\omega^p|_{F_M}=g(x)\,\HH^n|_{F_M}$ with $g(x)\geq M$, we have by \eqref{Frostman}
\begin{align*}
\frac12\,\omega^p(F_M)&\leq \frac{1}{2}=\frac12\,\mu(F_M)\leq \mu\Bigl(F_M\setminus \bigcup_{Q\in\LM} Q\Bigr)\\
&\leq C\,\HH^n\Bigl(F_M\setminus \bigcup_{Q\in\LM} Q\Bigr)\leq C\,M\,\omega^p\Bigl(F_M\setminus \bigcup_{Q\in\LM} Q\Bigr)
.
\end{align*}

On the other hand, since $\Theta^{n,*}(x,\omega^p):=\limsup_{r\to0}\frac{\omega^p(B(x,r))}{(2r)^n}<\infty$ for $\omega^p$-a.e. $x\in\R^{n+1}$, it is also clear that
for $A$ big enough
$$\omega^p\biggl(\,\bigcup_{Q\in\HD}Q\biggr)\ll  \omega^p(F_M).$$
From the above estimates it follows that
\begin{equation}\label{eqbig}
\omega^p\biggl(F_M\setminus \bigcup_{Q\in\bad} Q\biggr) >0
\end{equation}
if $\tau$ and $A$ have been chosen appropriately. 
\vv

For technical reasons we have now to introduce a variant of the family $\DD^{db}$ of doubling cubes
defined in Section \ref{secdya}.
Given some constant $T\geq C_0$ (where $C_0$ is the constant in Lemma \ref{lemcubs}) to be fixed below,
we say that $Q\in\wt\DD^{db}$ if
$$
\omega^p(100B(Q))\leq T\,\omega^p(Q).
$$
We also set $\wt \DD^{db}_k=\wt\DD^{db}\cap \DD_k$ for $k\geq0$.
From \rf{eqdob*} and the fact that $T\geq C_0$, it is clear that $\DD^{db}\subset \wt\DD^{db}$.

\vv

\begin{lemma}\label{lemgg}
 If the constant $T$ is chosen big enough, then 
$$\omega^p\biggl(F_M \cap \bigcup_{Q\in\wt\DD_0^{db}} Q  \setminus
 \bigcup_{Q\in\bad} Q
\biggr) >0.$$
\end{lemma}

Notice that above $\wt\DD_0^{db}$ stands for the family of cubes from the zero level of $\wt\DD^{db}$.

\begin{proof}
By the preceding discussion we already know that 
$$\omega^p\biggl(F_M\setminus \bigcup_{Q\in\bad} Q\biggr) >0.$$
If $Q\not\in\wt\DD^{db}$, then $\omega^p(Q)\leq T^{-1}\omega^p(100B(Q))$. Hence by the finite overlap of the balls
$100B(Q)$ associated with cubes from $\DD_0$ we get
$$\omega^p\biggl(\,\bigcup_{Q\in\DD_0\setminus\wt\DD^{db}} Q\biggr) \leq \frac1T\sum_{Q\in\DD_0}\omega^p(100B(Q))\leq
\frac CT\,\|\omega^p\|=\frac CT.$$
Thus for $T$ big enough we derive
$$\omega^p\biggl(\,\bigcup_{Q\in\DD_0\setminus\wt\DD^{db}} Q\biggr) \leq \frac12\,\omega^p\biggl(F_M\setminus \bigcup_{Q\in\bad} Q\biggr),$$
and then the lemma follows.
\end{proof}
\vv

Notice that for the points $x\in F_M\setminus \bigcup_{Q\in\bad} Q$, from the condition (b) in the definition
of bad cubes, it follows that
$$\omega^p(B(x,r))\lesssim A\,r^n\qquad \mbox{for all $0<r\leq 1$.}$$
Trivially, the same estimate holds for $r\geq1$, since $\|\omega^p\|=1$. So we have
\begin{equation}\label{eqcc0}
M^n\omega^p(x)\lesssim A\quad \mbox{ for $\omega^p$-a.e.\ $x\in F_M\setminus \bigcup_{Q\in\bad} Q$.}
\end{equation}

\vv

\section{The key lemma about the Riesz transform of $\omega^p$ on the good cubes}

\begin{lemma}[Key lemma]
Let $Q\in\good$ be contained in some cube from the family $\wt{\DD}_0^{db}$, and $x\in B_Q$. Then we have
\begin{equation}\label{eqdk0}
\bigl|\RR_{r(B_Q)}\omega^p(x)\bigr| \leq  C(A,M,T,\tau,d_p),
\end{equation}
where, to shorten notation, we wrote $d_p= \dist(p,\partial\Omega)$.
\end{lemma}

\begin{proof} 
To prove the lemma, clearly we may assume that $r(B_Q)\ll\dist(p,\partial\Omega)$ and that $r(P)<r_{0}$ for any $P\in \good$, where $r_{0}$ is as in the statement of Theorem \ref{teo1}.
First we will prove \rf{eqdk0} for $Q\in\wt\DD^{db}\cap\good$.
In this case, by definition we have
$$\mu(Q)> \tau\,\omega^p(Q)\quad \mbox{and}\quad \omega^p(3B_Q)\leq T\,\omega^p(Q).$$
We consider a ball $B$ centered on $Q\cap \supp\mu$ with $\delta^{-1}B\subset 2B_Q$ (where $\delta$ is the constant in
Lemma \ref{lembourgain}) such that $\mu(B)\gtrsim \mu(Q)$ and $r(B)\sim_\delta r(B_Q)$. Also, appealing to the porosity condition 
of $\partial\Omega$ in $E$ and the fact that $\supp\mu\subset E$, we may take another ball $B_0$ such that $\overline{B_0}\subset B\setminus \partial\Omega$ with $$r(B_0)\sim r(B)\sim r(B_Q).$$
Here (as well as in the rest of the lemma) all implicit constants may depend on $\delta$.

Denote by $\EE(x)$ the fundamental solution of the Laplacian in $\R^{n+1}$, so that the Green function $G(\cdot,\cdot)$
of $\Omega$ equals
\begin{equation}\label{eqgreen}
G(x,p) = \EE(x-p) - \int \EE(x-y)\,d\omega^p(y).
\end{equation}
Notice that the kernel of the Riesz transform is
\begin{equation}\label{eqker}
K(x) = c_n\,\nabla \EE(x),
\end{equation}
for a suitable absolute constant $c_n$. 
For $x\in\R^{n+1}\setminus \overline\Omega$, since $K(x-\cdot)$ is harmonic in $\Omega$, we have
\begin{equation}\label{eqclau2}
\RR\omega^p(x) = \int K(x-y)\,d\omega^p(y) = K(x-p).
\end{equation}
For $x\in\Omega$, by \rf{eqker} and  \rf{eqgreen}  we have
\begin{align}\label{eqclau1}
\RR\omega^p(x) = 
c_n\nabla_x\int \EE(x-y)\,d\omega^p(y) & = c_n\,\nabla_x\bigl(\EE(x-p) - G(x,p)\bigr) \nonumber\\
&= K(x-p) - c_n\,\nabla_x G(x,p).
\end{align}

So if  $B_0\subset \R^{n+1}\setminus \overline\Omega$, then \rf{eqclau2} holds for all $x\in B_0$, while
 if $B_0\subset \Omega$, then 
 every $x\in B_0$ satisfies \rf{eqclau1}.
We claim that, in any case, for the center $z_{B_0}$ of $B_0$ we have
\begin{equation}\label{eqclau0}
|\RR\omega^p(z_{B_0})|\lesssim 1.
\end{equation}
This is clear if $B_0\subset \R^{n+1}\setminus \overline\Omega$, since in this case
$$|\RR\omega^p(z_{B_0})|= |K(p-z_{B_0})|\sim \dist(p,\partial\Omega)^{-n}.$$
Suppose now that $B_0\subset \Omega$.
From \rf{eqclau1} we infer that for all $x\in B_0$ we have
\begin{equation}\label{eqkey6}
|\RR\omega^p(x)|^2 \lesssim 1 + |\nabla_x G(x,p)|^2.
\end{equation}
Averaging this with respect to the Lebesgue measure $m$ on $\frac14 B_0$ and applying Caccioppoli's inequality,
\begin{align}\label{eqdk1}\nonumber
\avint_{\frac14B_0} |\RR\omega^p|^2\,dm & \lesssim 
1 + \avint_{\frac14B_0} |\nabla_x G(x,p)|^2\,dm(x)\\
& \lesssim 
1 + \avint_{\frac12B_0} \frac{|G(x,p) - G(z_{B_0},p)|^2}{r(B_0)^2} \,dm(x) .
\end{align}
For $x\in\frac12 B_0$, in the case $n=1$,
by Lemma \ref{l:w>G} we have
$$|G(x,p) - G(z_{B_0},p)|\lesssim
\frac{\omega^p(\delta^{-1}B)}{r(\delta^{-1}B)^{n-1}}\,\frac1{\omega^{z_{B_0}}(\delta^{-1}B)}
$$
The same estimate holds for $n\geq2$ using that
$$|G(x,p) - G(z_{B_0},p)|\leq |G(x,p)|+ G(z_{B_0},p)|\lesssim  G(z_{B_0},p)|,$$
by Harnack's inequality, and then plugging Lemma \ref{l:w>G} again.
Also, from Bourgain's Lemma \ref{lembourgain} and \eqref{Frostman} we get
$$\omega^{z_{B_0}}(\delta^{-1}B) \geq \frac{\mu(B)}{r(B)^n}.$$
Therefore,  
$$|G(x,p) - G(z_{B_0},p)|
\lesssim \frac{\omega^p(\delta^{-1}B)}{r(\delta^{-1}B)^{n-1}}\,\frac{r(B)^n}{\mu(B)} \lesssim \frac{\omega^p(2B_Q)\,r(B)}{\mu(Q)}.$$
From the fact that $Q$ is doubling (from $\wt\DD^{db}$) and good, we deduce that $\omega^p(2B_Q)\lesssim \omega^p(Q)\leq\tau^{-1}\mu(Q)$, and so
$$|G(x,p) - G(z_{B_0},p)|\leq C(\tau)\,r(B)\quad\mbox{for all $x\in \frac12B_0$.}$$ 
Thus, by the harmonicity of $\RR\omega^p$ in $B_0$,  H\"older's inequality, \rf{eqdk1}, and the last estimate, we get
\begin{align*}
|\RR\omega^p(z_{B_0})| & = \left|\avint_{\frac14B_0} \RR\omega^p\,dm\right|\lesssim 
\avint_{\frac14B_0} |\RR\omega^p|^2\,dm\\
&\lesssim 1 + \avint_{\frac12B_0} \frac{|G(x,p) - G(z_{B_0},p)|^2}{r(B_0)^2}  \,dm(x) \lesssim 1
\end{align*}
with the implicit constant depending on $\tau$ and other parameters of the construction, and so \rf{eqclau0} holds in this case too.

From standard Calder\'on-Zygmund estimates and the fact that
$$|\RR_{r(B_0)}\omega^p(z_{B_0})| = |\RR\omega^p(z_{B_0})|\lesssim 1,$$
we derive that, for all $y\in B_Q$,
$$|\RR_{r(B_Q)}\omega^p (y)|\lesssim |\RR_{r(B_0)}\omega^p (z_{B_0})| + M^n_{\ell(Q)}\omega^p(z_{B_Q})\lesssim_A 1,$$
where $z_{B_Q}$ is the center of $B_Q$. In the last estimate we took into account that $Q$ and hence all its ancestors are good and thus $Q\not\in\HD$.
Hence the lemma holds when $Q\in\wt{\DD}^{db}\cap\good$.
\vv

Consider now the case $Q\in\good\setminus\wt\DD^{db}$. Let $Q'\supset Q$ be the cube from $\wt\DD^{db}$ with minimal side length.
The existence of $Q'$  is guarantied by the assumption in the lemma regarding the existence of some cube from $\wt\DD^{db}_0$ 
containing $Q$.
For all $y\in B_Q$ then we have
$$|\RR_{r(B_Q)}\omega^p (y)| \leq |\RR_{r(B_{Q'})}\omega^p (y)| + C\,\sum_{P\in\DD: Q\subset P\subset Q'}\Theta_\omega(2B_P). 
$$
The first term on the right hand side is bounded by some constant depending on $A,M,\tau,\ldots$.
To bound the last sum we can apply Lemma \ref{lemcad23} (because the cubes that are not from $\wt\DD^{db}$ do not belong to $\DD^{db}$ either) and then we get 
$$\sum_{P\in\DD: Q\subset P\subset Q'}\Theta_\omega(2B_P)\lesssim C\,\Theta_\omega(4B_Q').$$ 
Finally, since $Q'\not\in\HD$, we have $\Theta_\omega(4B_Q')\lesssim C\,A$. So \rf{eqdk0} also holds in this case.
\end{proof}
\vv

\vv
From the lemma above we deduce the following corollary.

\begin{lemma}\label{lemaxcor}
For $Q\in\good$ and $x\in B_Q$, we have
\begin{equation}\label{eqdk10}
\RR_{*,r(B_Q)}\omega^p(x) \leq  C(A,M,\tau,d_p),
\end{equation}
where, to shorten notation, we wrote $d_p= \dist(p,\partial\Omega)$.
\end{lemma}

\vv

\section{The end of the proof of Theorem \ref{teo1}}\label{secend}

Set
$$G= F_M \cap \bigcup_{Q\in\wt{\DD}_0^{db}} Q  \setminus
 \bigcup_{Q\in\bad} Q.$$
and recall that, by
Lemma \ref{lemgg}, 
$$\omega^p(G)>0.$$
As shown in \rf{eqcc0}, we have
\begin{equation}\label{eqcc1}
M^n\omega^p(x)\lesssim A\quad \mbox{ for $\omega^p$-a.e.\ $x\in G$.}
\end{equation} 
On the other hand, from Lemma \ref{lemaxcor} is also clear that
\begin{equation}\label{eqcc2}
\RR_{*}\omega^p(x)\leq C(A,M,\tau,d_p)\quad \mbox{ for $\omega^p$-a.e.\ $x\in G$.}
\end{equation}

Now we will apply the following result.

\begin{theorem}  \label{teo**}
Let $\sigma$ be a Radon measure with compact support on $\R^d$ and consider a $\sigma$-measurable set
$G$ with $\sigma(G)>0$ such that
$$G\subset\{x\in \R^d: 
M^n\sigma(x) < \infty \mbox{ and } \,\RR_* \sigma(x) <\infty\}.$$
Then there exists a Borel subset $G_0\subset G$ with $\sigma(G_0)>0$  such
that $\sup_{x\in G_0}M^n\sigma|_{G_0}(x)<\infty$ and  $\RR_{\sigma|_{G_0}}$ is bounded in $L^2(\sigma|_{G_0})$.
\end{theorem}

This result follows from the deep non-homogeneous Tb theorem of Nazarov, Treil and Volberg in \cite{NTV} (see also \cite{Volberg}) in combination with the methods in \cite{Tolsa-pams}. For the detailed proof in the case of the Cauchy
transform, see \cite[Theorem 8.13]{Tolsa-llibre}. The same arguments with very minor modifications work for the Riesz transform.
\vv

From \rf{eqcc1}, \rf{eqcc2} and Theorem \ref{teo**} applied to $\sigma=\omega^p$ in case that $\partial\Omega$ is compact, we infer that there exists a subset $G_0\subset G$
such that the operator $\RR_{\omega^p|_{G_0}}$ is bounded in $L^2(\omega^p|_{G_0})$. By Theorem 1.1 of \cite{NToV-pubmat}
(or the David-L\'eger theorem  \cite{Leger} for $n=1$), we deduce
that $\omega^p|_{G_0}$ is $n$-rectifiable.

If $\partial\Omega$ is non-compact, then we consider a ball $B(0,R)$ such that $\omega^p(G\cap B(0,R))>0$ and we set
$\sigma = \chi_{B(0,2R)}\omega^p$. Since 
$$\RR_*(\chi_{B(0,2R)^c}\omega^p)(x)\leq \frac{\omega^p(B(0,2R)^c)}{R}<\infty \quad \mbox{ for all $x\in B(0,R)$,}$$
from \rf{eqcc2} we infer that 
$\RR_*\sigma(x)<\infty$ for all $x\in G\cap B(0,R)$, and so we can argue as above.

\vvv

\end{document}